\documentclass[a4paper,12pt]{amsart}

\usepackage{amssymb}
\usepackage{amsmath}
\usepackage{mathtools}
\usepackage{xcolor} 
\usepackage{fancybox}
\usepackage{dashbox}
\usepackage{esint} %% for fint command
\newboolean{showcomments}
\setboolean{showcomments}{true}
\ifthenelse{\boolean{showcomments}}
{ \newcommand{\mynote}[3]{
  \fbox{\bfseries\sffamily\scriptsize#1}
  {\small$\blacktriangleright$\textsf{\emph{\color{#3}{#2}}}$\blacktriangleleft$}}}
{ \newcommand{\mynote}[3]{}}
\definecolor{asparagus}{rgb}{0.53, 0.66, 0.42}

\usepackage{mathrsfs}

  \usepackage{float}

  \usepackage{pgffor}
  \usepackage{tikz}
  \usetikzlibrary{calc,3d,arrows,positioning,shapes.misc}

  \usepackage[english]{babel}
  \usepackage{amsthm}
  \usepackage{amsmath}
  \usepackage{amssymb}
  \usepackage[latin1]{inputenc}
  \usepackage{setspace}
  \usepackage{enumerate}
  \usepackage{stmaryrd}

  \usepackage[colorlinks=true,linkcolor=black,citecolor=black,filecolor=black,urlcolor=black,menucolor=black]{hyperref}
  \usepackage{array}
  \usepackage{graphicx}
  \usepackage[babel]{csquotes}
  \usepackage{geometry}
  \usepackage{bbm}
  \usepackage{stmaryrd}
  \usepackage[all]{xy}
  \usepackage{mathrsfs}%fÃ¼r mathscr

  \theoremstyle{plain}
  \newtheorem{theorem}{Theorem}%[section]
  \newtheorem{proposition}{Proposition}

  \newtheorem{lemma}{Lemma}
  
  \theoremstyle{definition}
  \newtheorem{definition}{Definition}

	\allowdisplaybreaks[3]

  \def \E{\mathcal{E}}
  \def \F{\mathcal{F}}
  \def \H{\mathcal{H}}

  \def \R{\mathbb{R}}
  
  \def \transpose{\mathsf{T}}

  \newcommand {\supp} {\mathop \textup{supp}}
  
 \renewcommand {\div} {\mathrm{div}}
 
 \newcommand {\dist} { \mathop \textup{dist} \nolimits}
\newcommand{\p}{\partial}
\newcommand{\dL}{\,\mathrm{d}}
\newcommand{\ddt}{\frac{\mathrm{d}}{\mathrm{d}t}}
  \newcommand {\eps} {\varepsilon}

  \renewcommand{\{}{\left\lbrace}
    \renewcommand{\}}{\right\rbrace}

  \newcommand{\sdist}{\mathbf{s}}

\title[Quantitative convergence of the nonlocal Allen--Cahn equation]{Quantitative convergence of the nonlocal Allen--Cahn equation to volume-preserving mean curvature flow}

\author{Milan Kroemer}
\author{Tim Laux}
\address{Milan Kroemer, Tim Laux, Hausdorff Center for Mathematics, University of Bonn,  Villa Maria, Endenicher Alllee 62, 53115 Bonn, Germany}
\email{\string{milan.kroemer,tim.laux\string}@hcm.uni-bonn.de}

\begin{document}

    \begin{abstract}
	  We prove a quantitative convergence result of the nonlocal Allen--Cahn equation to volume-preserving mean curvature flow.
    The proof uses gradient flow calibrations and the relative entropy method,
    which has been used in the recent literature to prove weak-strong uniqueness results for mean curvature flow and convergence of the Allen--Cahn equation. 
    A crucial difference in this work is a new notion of gradient flow calibrations.
    We add a tangential component to the velocity field in order to prove the Gronwall estimate for the relative energy.
    This allows us to derive the optimal convergence rate without having to show the closeness of the Lagrange-multipliers.
	
	\medskip
    

  \medskip

  \end{abstract}

\maketitle

%\tableofcontents

%%%%%%%%%%%%%%%%%%%%%%%%%%%%%%%%%%%%%%%%%%%%%%%%%%%%%%%%%%%%%%%%%%%%%%%%%%%%%%%%%%%%%%%%%%%%%%%%%%%%%%%%%%%%%%%%

\section{Introduction}

	% References:
	
	% Rubinstein-Sternberg \cite{RubinsteinSternberg}
	% Laux-Simon \cite{LauxSimon}
	% Fischer-Laux-Simon \cite{FischerLauxSimon}
	% Laux-Liu \cite{LauxLiu},
	% Fischer-Marveggio \cite{FischerMarveggio} 
	% Hensel-Laux (contact angle) \cite{HenselLauxContact}
	% Hensel-Moser (contact angle) \cite{HenselMoserContact}
	% Hensel-Laux (De Giorgi) \cite{HenselLauxVarifold}
	% Fischer-Hensel-Laux-Simon \cite{FHLS}
	% Takasao \cite{Takasao}, \cite{TakasaoHigherD}
	% Laux-Otto, \cite{LauxOtto, LauxOttoBrakke, LauxOttoDeGiorgi}
	% Mugnai-Seis-Spadaro \cite{MugnaiSeisSpadaro}
	% Laux-Swartz (thresholding) \cite{LauxSwartz}
	% Kim-Kwon \cite{KimKwon}
	% Kim-Kwon-Po\v{z}\'{a}r \cite{KimKwonPozar} (vol pres mcf via comparison principle)
	% Chen-Hilhorst-Logak \cite{ChenHilhorstLogak}
	% Gage \cite{GageArea}
	% Huisken \cite{HuiskenVol}
	
	We consider the nonlocal Allen--Cahn equation
  \begin{align}\label{eq:AC}
		\partial_t u_\eps = \Delta u_\eps - \frac1{\eps^2} W'(u_\eps) + \lambda_\eps\sqrt{2W(u_\eps)}
	\end{align}
  which was first introduced by Golovaty~\cite{Golovaty}.
  Here $\lambda_\eps=\lambda_\eps(t)$ is a Lagrange multiplier which is given explicitly by
	\begin{align}\label{eq:lambdaeps}
		\lambda_\eps(t) \coloneqq -\frac{\int_{\R^d} (\Delta u_\eps -\frac1{\eps^2} W'(u_\eps)) \sqrt{2W(u_\eps)}\dL x}{\int_{\R^d} 2W(u_\eps) \dL x}.
	\end{align} 
  This is a natural choice since then the mass of $\psi_\eps\coloneqq\phi \circ u_\eps$,
  where $\phi(u) =\int_0^{u} \sqrt{2W(z)}\dL z$, is preserved:
	\begin{align}
		\ddt \int (\phi \circ u_\eps)(x,t)\dL x = \int_{\R^d} \sqrt{2W(u_\eps(x,t))} \p_t u_\eps(x,t) \dL x =0.
	\end{align}
  This change of variables $\phi:u_\eps\mapsto\psi_\eps$ is crucial in studying the Allen--Cahn equation
  and was discovered by Modica--Mortola~\cite{modicamortola} and independently by Bogomol'nyi~\cite{bogomolnyi}.
  In the present paper, we derive an optimal quantitative convergence result in the sharp interface limit, see Theorem~\ref{thm:ACtoMCF} below.

  Nonlocal versions of the Allen--Cahn equation were first introduced by Rubinstein and Sternberg~\cite{RubinsteinSternberg}
  as a basic model for coarsening processes which conserve the phase volume. 
	The original model by Rubinstein and Sternberg is
	\begin{align*}
		\partial_t u_\eps =  \Delta u_\eps - \frac1{\eps^2} W'(u_\eps) + \frac1\eps \lambda_\eps,
	\end{align*} 
	where $\lambda_\eps =\lambda_\eps(t)$ is the Lagrange multiplier associated to the mass constraint\linebreak
  $\int_{\R^d} u_\eps(x,t) \dL x = \int_{\R^d} u_\eps(x,0)\dL x$
	and is explicitly given by $\lambda_\eps(t) = \int_{\R^d} \frac1\eps W'(u_\eps(x,t)) \dL x$.
  Equation~\eqref{eq:AC} has several advantages over the classical Rubinstein--Sternberg model
  as the effect of the Lagrange multiplier is amplified close to the diffuse interface.

	The nonlocal Allen--Cahn equation~\eqref{eq:AC} is the $L^2$-gradient flow of the Cahn--Hilliard energy
	\begin{align}
	E_\eps[u] = \int_{\R^d} \left( \frac\eps2 |\nabla u|^2 +\frac1\eps W(u) \right) \dL x
	\end{align}
	restricted to the mass-constrained ``submanifold''~$\{  u \colon \int_{\R^d} \phi \circ u\dL x = m\} \subset L^2(\R^d)$
  and sped up by the factor $\frac1\eps$.
	This gradient-flow structure can be read off from the optimal energy dissipation relation which holds for any classical solution of~\eqref{eq:AC}:
	\begin{align}
	\ddt E_\eps[u_\eps(\cdot,t)] = -\int_{\R^d} \eps (\p_t u_\eps(x,t))^2 \dL x.
	\end{align}

	The investigation of the sharp-interface limit $\eps\to0$ of nonlocal versions
  of the Allen--Cahn equation~\eqref{eq:AC}--\eqref{eq:lambdaeps} started with the matched asymptotic expansion by Golovaty~\cite{Golovaty}. 
	His formal argument suggests that the limit evolves by the nonlocal evolution equation
	\begin{align}\label{eq:introVPMCF}
		V=-H+\lambda \quad \text{on $\Sigma(t)$,}
	\end{align}
	where $V$ and $H$ denote the normal velocity and the mean curvature of the evolving surface $\Sigma(t)=\p \Omega(t)$, respectively,
  and $\lambda = \lambda(t)$ is the Lagrange multiplier corresponding to the volume constraint $|\Omega(t)|=|\Omega(0)|$.
	Also this equation, the volume-preserving mean curvature flow, has a gradient-flow structure as is seen at the energy dissipation relation
	\begin{align}
		\ddt E[\Sigma(t)]
		= \int_{\Sigma(t)} V(x,t)H(x,t) \dL\H^{d-1}(x) 
		=- \int_{\Sigma(t)} V^2 \dL\H^{d-1}(x),
	\end{align}
	which holds for sufficiently regular solutions of~\eqref{eq:introVPMCF}. 
	Again the evolution is restricted to a ``submanifold'' $\{\Sigma=\partial \Omega \subset \R^d \colon |\Omega| = m\}$
  which incorporates the volume constraint.
	Takasao showed under very mild assumptions that solutions to~\eqref{eq:AC}--\eqref{eq:lambdaeps} converge
  to a weak solution of volume-preserving mean curvature flow in the sense of Brakke~\cite{Brakke};
	first for ambient dimensions $d=2,3$~\cite{Takasao} and most recently, for a slight perturbation of~\eqref{eq:AC}--\eqref{eq:lambdaeps}
  in all dimensions~\cite{TakasaoHigherD}.
	Another approach is inspired by the work of Luckhaus and Sturzenhecker~\cite{LucStu}:
  the second author and Simon~\cite{LauxSimon} showed that, under a natural energy-convergence assumption
	as in~\cite{LucStu}, the limit is a distributional solution to volume-preserving mean curvature flow,
  which holds in all spatial dimensions and also in the case of multiple phases, 
	any selection of which may carry a volume constraint.

  For our proof, we use the relative energy method.
  In the context of the convergence of phase field models this method was introduced by Fischer, Simon and the second author in~\cite{FischerLauxSimon},
  but the relative energy is very closely related to the diffuse tilt-excess introduced by Simon and the second author in~\cite{LauxSimon}.
  It can also be used to incorporate boundary contact, as was shown by Hensel and Moser~\cite{HenselMoserContact}, and Hensel and the second author~\cite{HenselLauxContact}.
  As the method does not rely on the maximum principle, it can also be applied for vectorial problems.
  Liu and the second author~\cite{LauxLiu} combined the relative energy method with weak convergence methods to derive the scaling limit of transitions
  between the isotropic and the nematic phase in liquid crystals.
  Fischer and Marveggio~\cite{FischerMarveggio} showed that the method can also be used for the vectorial Allen--Cahn equation, at least in ambient dimensions $d=2,3$
  and for a prototypical potential with three wells.

  The nonlocal Allen--Cahn equation is a physically motivated model, which is why its sharp interface limit is of high interest.
  But it can also be viewed as an approximation scheme to construct (numerically or theoretically) solutions to volume preserving mean curvature flow.
  Other methods to construct solutions include PDE methods which can be used for short time~\cite{EscherSimonett};
  versions of the minimizing movements scheme by Almgren, Taylor and Wang~\cite{Almgren1993}, as was first done by Mugnai, Seis and Spadaro~\cite{MugnaiSeisSpadaro}
  and later by Julin and Niinikoski~\cite{Julin2022};
  and the thresholding scheme, which is also numerically efficient, see the work of Swartz and the second author~\cite{LauxSwartz}.

	\subsection{Notation}
		The Landau symbol $O$ will be used frequently. Precisely, by $a=O(b)$ we mean that there exists a constant $C$
    depending on $d$, $T$, and $\Sigma=(\Sigma(t))_{t\in[0,T]}$, such that $|a| \leq C |b|$.
    The signed distance function to $\Sigma(t)$ will be denoted by
    \begin{align}\label{eq:defsdist}
      \sdist(x,t)\coloneqq\dist(x,\Omega(t))-\dist(x,\mathbb{R}^d\setminus\Omega(t)),
    \end{align}
    where $\Omega(t)$ is the region enclosed by $\Sigma(t)$.
    The gradient and divergence on $\mathbb{R}^d$ will be denoted by $\nabla$ and $\div$, respectively.
    In the neighborhood of a surface $\Sigma$ the tangential gradient and divergence will be denoted by $\nabla_{\Sigma}$ and $\div_{\Sigma}$, 
    and are explicitly given by
    \[\nabla_\Sigma=(\mathrm{Id}-\nu\otimes\nu)\nabla\quad\text{and}\quad\div_\Sigma=(\mathrm{Id}-\nu\otimes\nu):\nabla.\]
    These operators can also be defined intrinsically on $\Sigma$ so that we can apply them to functions and vector fields only defined on the surface.
		
	\section{Main results}

  The main result of this work states that solutions to the nonlocal Allen--Cahn equation with well-prepared initial conditions
  converge to solutions of volume-preserving mean curvature flow before the onset of singularities.
	In addition, the theorem provides the optimal convergence rate $O(\eps)$.
	For simplicity we assume that the two wells of $W$ are $0$ and $1$ and that the induced surface tension is normalized
  to $\sigma \coloneqq \phi(1)=\int_0^1 \sqrt{2W(z)} \dL z =1$.
	This is for example the case if $W(z)= 18 z^2 (z-1)^2$.
	
	\begin{theorem}\label{thm:ACtoMCF}
    Let $\Sigma=(\Sigma(t)=\partial\Omega(t))_{t\in [0,T]}$ be a smooth solution to volume-preserving mean curvature flow
    according to Definition~\ref{def:strong} below
    and let $u_\eps$ be a solution of the nonlocal Allen--Cahn equation~\eqref{eq:AC}
    with well-prepared initial conditions according to Definition~\ref{def:wellprepared} below.
    Then there exists a constant $C=C(d,\Sigma,T)<\infty$ such that
		\begin{align*}
      \sup_{t\in[0,T]}\int_{\mathbb{R}^d}|\psi_\eps(x,t)-\chi_{\Omega(t)}(x)| \dL x \leq C\eps.
		\end{align*}
	\end{theorem}
  
  We note that well-prepared initial data can easily be constructed by gluing the optimal profile around $\Sigma(0)$:
  
  \begin{lemma}\label{lem:existence_init_data}
    If $\Sigma(0)$ is $C^{3,\alpha}$ for some $\alpha\in(0,1)$,
    then there exist constants $(a_\eps)_{\eps>0}$ with $a_\eps=O(\eps)$ as $\eps\downarrow 0$ such that
    \begin{align}\label{eq:init_data_optimal_profile}
      u_{\eps}(x,0)\coloneqq U\left(\frac{-\sdist(x,\Sigma(0))-a_\eps}{\eps}\right)
      \quad\text{is well-prepared in the sense of Definition~\ref{def:wellprepared},}
    \end{align}
    where $U$ is the unique solution to $U''=W'(U)$ with
    $U(-\infty)=0,\,U(+\infty)=1$ and $U(0)=\frac{1}{2}$.
  \end{lemma}

 \begin{definition}\label{def:strong}
    We call a family of surfaces $\Sigma=(\Sigma(t))_{t\in[0,T]}$ a smooth solution to volume-preserving mean curvature flow
    if there exists $\alpha\in(0,1)$ such that
    $\Sigma(t)$ is $C^{3,\alpha}$ for all $t$ and $\Sigma(t)$ evolves by~\eqref{eq:introVPMCF}, i.e., $V=-H+\lambda$,
    and the normal velocity $V(t)$ is of class $C^{1,\alpha}$ in space.
  \end{definition}
 
  Before we give a precise definition of well-preparedness, we need to introduce some definitions.
  The key tool in our proof is a suitable gradient flow calibration.

	\begin{definition}\label{def:GF_Cal}
		Let $\Sigma=(\Sigma(t))_{t\in[0,T]}$ be a one-parameter family of closed surfaces $\Sigma(t) = \partial \Omega(t) \subset \R^d$. 
		Let $\xi,B \colon \R^d\times[0,T]\to \R^d$ be two vector fields, let $\vartheta \colon \R^d\times[0,T]\to \R$ and let $ \lambda \colon [0,T]\to \R$.
		We call the tuple $(\xi,B, \vartheta, \lambda)$ a \emph{gradient-flow calibration for volume-preserving mean curvature flow} if the following statements hold true.
		\begin{enumerate}[(i)]
			\item \emph{Regularity.}\label{item:reg}
			The vector field $\xi$ and the function $\vartheta$ satisfy
			\begin{align}
				\xi \in 
				C^{0,1}(\R^d\times[0,T];\R^d) \quad\text{and} \quad  \vartheta \in C^{0,1}(\R^d\times[0,T]).
			\end{align}
			Furthermore, for each $t\in[0,T]$ it holds
			\begin{align}
				B(\cdot,t) \in C^{0,1}(\R^d;\R^d).
			\end{align}
			\item \emph{Normal extension and shortness.}\label{item:normal}
			The vector field $\xi$ extends the exterior unit normal vector field
			\begin{align}
				\xi(\cdot,t) = \nu(\cdot,t) \quad \text{on $\Sigma(t)$}
			\end{align}
			and it is short away from $\Sigma$: There exists a constant $c>0$ such that
			\begin{align}\label{eq:xishort}
				|\xi(\cdot,t)| \leq (1-c\dist^2(x,\Sigma(t)))_+,
			\end{align}
			where $(\cdot)_+$ denotes the positive part.
		\item \emph{Divergence constraint.} \label{item:divB}
        There exist a bounded function $c:[0,T]\rightarrow\mathbb{R}$ such that the vector fields $B(\cdot,t)$
      satisfy, for each $t\in [0,T]$,
			\begin{align}\label{eq:divB}
				\nabla \cdot B (\cdot,t)-c(t)=O\big(\dist(\cdot, \Sigma(t))\big),
			\end{align}
      and
      \begin{align}\label{eq:xixi_nablaB}
        \xi\otimes\xi:\nabla B(\cdot,t)=O(\dist(\cdot,\Sigma(t))).
      \end{align}
			\item \emph{Approximate transport equations.}\label{item:transport}
			The weight $\vartheta$ is transported to first order
			\begin{align}\label{eq:transp_weight}
					\left(\p_t \vartheta + (B\cdot \nabla)\vartheta \right)(\cdot,t) = O\big(\dist(\cdot,\Sigma(t))\big),
			\end{align}
			and the length of $\xi$ to second order
			\begin{align}\label{eq:transp_absxi}
				\left(\p_t |\xi|^2 + (B\cdot \nabla) |\xi|^2\right)(\cdot,t) = O\big(\dist^2(\cdot,\Sigma(t))\big).
			\end{align}
			Furthermore
			\begin{align}\label{eq:transp_xi}
				\left(\p_t \xi + (B\cdot \nabla ) \xi + (\nabla B)^{\transpose} \xi \right)(\cdot,t)
				=O\big(\dist(\cdot,\Sigma(t))\big).
			\end{align}
%			Finally, we also have
%			\begin{align}
%				\p_t \vartheta +(B\cdot \nabla) \vartheta = O(\vartheta).
%			\end{align}
			\item \emph{Geometric evolution equation.}\label{item:GEE}
			
			\begin{align}\label{eq:extGEE}
				B(\cdot,t)\cdot\xi(\cdot,t)+\nabla \cdot \xi(\cdot,t)-\lambda(t)=O\big(\dist(\cdot,\Sigma(t))\big).
			\end{align}
			
			\item \emph{Coercivity of the transported weight.} \label{item:signweights}
        It holds
        \begin{align*}
          \vartheta(\cdot,t)&>0\quad\text{on $\mathbb{R}^d\setminus\Omega(t)$,} \\
          \vartheta(\cdot,t)&<0\quad\text{in $\Omega(t)$,} \\
		  \sup_{(x,t)\in\mathbb{R}^d\times[0,T]}|\vartheta(x,t)|&<\infty,
        \end{align*}
		and there exist constants $0<c,C<\infty$ such that, on $\supp\xi$,
		\begin{align}\label{eq:weight-bilip-estimate}
          c\dist(\cdot,\Sigma(t))\leq|\vartheta(\cdot,t)|\leq C\dist(\cdot,\Sigma(t)).
		\end{align}
			
%			\item \emph{Sign condition}
%				\begin{align}
%					s(\cdot,t)\vartheta(\cdot,t) >0  \quad \text{in } \R^d \setminus \Sigma(t).
%				\end{align}
		\end{enumerate}
	In case such a gradient-flow calibration exists for $\Sigma$, we call $\Sigma$ a \emph{calibrated flow}.
	\end{definition}

  The main difficulty in this work, compared to previous works using relative energy methods,
  are the divergence constraints~\eqref{eq:divB} and~\eqref{eq:xixi_nablaB} on $B$ which need a particular construction.
  These divergence constraints are natural in the following sense. 
  In view of~\cite{Laux2022}, it is useful to choose $B$ such that its divergence is controlled, since $\nabla\cdot B=0$
  is the localized version of the preservation of the total volume.
  There, it was chosen such that $\nabla \cdot B =O(\dist(\cdot,\Sigma))$.
  Here, we need to relax this constraint to (10) as we additionally want to fix the $\nu\otimes \nu$ component of the Jacobian $\nabla B$.
  Then $\nabla\cdot B = (I-\nu \otimes \nu ) \colon \nabla B = \div_{\Sigma} B$.
  And since the rate of change of the total surface area is dictated by the PDE, we cannot set $c(t)=0$.
  
  Our ansatz is to add a tangential part to the velocity field, say $X$.
  Then $B=V\nu+X$ on $\Sigma$ and the divergence constraint $\nabla\cdot B=c$ on $\Sigma$ becomes
  \[\div_{\Sigma}X=V H-c.\]
  Hence we see that necessarily
  \begin{align}
    c(t) &= \frac{\int_{\Sigma} VH\dL\H^{d-1}}{\H^{d-1}(\Sigma)}.
  \end{align}
  This PDE is underdetermined, so we make the ansatz that $X$ is a gradient field, i.e., $X=\nabla_{\Sigma}\varphi$
  for some potential $\varphi$.
  Then $\varphi$ solves the Poisson equation
  \begin{align}
    \Delta_{\Sigma}\varphi&=V H-c\quad\text{on $\Sigma$,}
  \end{align}
  where $\Delta_{\Sigma}=\div_{\Sigma}\nabla_{\Sigma}$ is the Laplace--Beltrami operator on $\Sigma$.

  Now Theorem~\ref{thm:ACtoMCF} rests on the following two propositions.
  The first one guarantees the existence of a calibration, the second shows that, given a calibration,
  the Allen--Cahn equation converges.

  \begin{proposition}\label{prop:strong_is_calibrated}
    If $\Sigma$ is a smooth solution to volume-preserving mean curvature flow in the sense of Definition~\ref{def:strong},
    then $\Sigma$ is a calibrated flow.
  \end{proposition}
  
	\begin{proposition}\label{prop:AC}
			Let $u_\eps$ be a solution to the nonlocal Allen--Cahn equation~\eqref{eq:AC}
      and let $(\Sigma(t))_{t\in[0,T]}$ be a calibrated flow according to Definition~\ref{def:GF_Cal}.
      Suppose further that\linebreak $\int_{\mathbb{R}^d}\psi(x,0)\dL x=|\Omega(0)|$.
			Then there exists a constant $C=C(d,T,\Sigma)$ such that, for all $t\in(0,T)$, it holds
			\begin{align}\label{eq:ddtEeps}
				\ddt \big(\E_\eps(t) +\F_\eps(t)\big) \leq C\big(\E_\eps(t)+\F_\eps(t)\big),
			\end{align}
      where $\E_\eps$ and $\F_\eps$ are defined below in~\eqref{eq:defEeps} and~\eqref{eq:defFeps}, respectively.
		\end{proposition}
	
		We work with the relative energy
    \begin{align}\label{eq:defEeps}
			\E_\eps(t)\coloneqq\E_\eps[u_\eps,\Sigma](t) 
			\coloneqq&E_\eps[u_\eps(\cdot,t)] + \int_{\R^d} \xi(x,t) \cdot \nabla \psi_\eps(x,t) \dL x
			\\\notag=& \int_{\R^d} \left( \frac\eps2 |\nabla u_\eps(x,t)|^2 + \frac1\eps W(u_\eps(x,t)) - |\nabla \psi_\eps(x,t)|\right)\dL x
			\\\notag&+ \int_{\R^d} \left( 1 - \xi(x,t) \cdot \nu_\eps(x,t) \right) |\nabla \psi_\eps(x,t)| \dL x,
		\end{align}
		where $\psi_\eps(x,t) \coloneqq  \int_0^{u_\eps(x,t)} \sqrt{2W(z)} \dL z$ and $\nu_\eps(x,t) \coloneqq -\frac{\nabla \psi_\eps(x,t)}{|\nabla \psi_\eps(x,t)|}$
		if $\nabla \psi_\eps(x,t)\neq0$ and $\nu_\eps(x,t)\coloneqq  e$ for some arbitrary $e\in S^{d-1}$ if $\nabla \psi_\eps(x,t)=0$.
		It is already clear that the relative energy $\E_\eps$ controls both the discrepancy between the two terms in the energy and the tilt-excess.
				
		Furthermore, we define the volume error functional
    \begin{align}\label{eq:defFeps}
			\F_\eps(t)\coloneqq\F_\eps[u_\eps,\Sigma](t) 
			&\coloneqq \int_{\R^d}  | \psi_\eps(x,t)- \chi_{\Omega(t)}(x)| |\vartheta(x,t)| \dL x
			\\\notag &= \int_{\R^d} ( \psi_\eps(x,t)- \chi_{\Omega(t)}(x))\vartheta(x,t) \dL x .
		\end{align}

    \begin{definition}\label{def:wellprepared}
     We call initial conditions $u_\eps(\cdot,0)$ \emph{well-prepared} if they satisfy the following assumptions:
    \begin{enumerate}[(i)]
      \item \emph{Mass constraint.} \label{item:mass} $\int_{\R^d} \psi_\eps(x,0) \dL x = |\Omega(0)|$.
      \item \emph{Optimal convergence rate.} \label{item:optimalrate} $\mathcal{E}_\eps(0)+\mathcal{F}_\eps(0)=O(\eps^2)$.
    \end{enumerate} 
    \end{definition}
  	
    The proofs of Proposition~\ref{prop:strong_is_calibrated} and~\ref{prop:AC} are deferred to the next sections.
    Now, based on the propositions we are able to prove Theorem~\ref{thm:ACtoMCF} similarly to~\cite{FischerLauxSimon}.
			
		\begin{proof}[Proof of Theorem~\ref{thm:ACtoMCF}]
			By Gronwall's lemma,~\eqref{eq:ddtEeps} implies
      \begin{align}\label{eq:pf_gronwall}
				\E_\eps(t)+\F_\eps(t)
				\leq C \big(\E_\eps(0)+\F_\eps(0) \big)
        \quad\text{for all $t\in[0,T]$.}
			\end{align}
      Now, for $\delta>0$ and $f\in L^\infty(0,r)$, we split the square $[0,r]^2$ into two triangles and apply Fubini's theorem
      \begin{align*}
        \left(\int_0^\delta|f(r)|\dL r\right)^2
        \leq 2\|f\|_\infty\int_0^\delta|f(r)|r\dL r.
      \end{align*}
      Let $\mathcal{U}_r(t)\coloneqq\{x:\dist(x,\Sigma(t))<r\}$ denote the tubular neighborhood of $\Sigma(t)$ with radius $r$
      and let $\pi_{\Sigma(t)}\coloneqq\mathrm{Id}-\sdist\nabla\sdist\otimes\nabla\sdist$ denote the orthogonal projection onto $\Sigma(t)$,
      where $\sdist$ is the signed distance function defined in~\eqref{eq:defsdist}.
      Now let $\delta>0$ sufficiently small such that $\pi_{\Sigma(t)}$ is well defined on $\mathcal{U}_\delta(t)$ and injective
      for all $t\in[0,T]$.
      We compute
      \begin{align*}
        &\left(
          \int_{\mathcal{U}_\delta(t)}|\psi_\eps(\cdot,t)-\chi_{\Omega(t)}|\dL x
        \right)^2 \\
        \leq\,\,&C\bigg(
          \int_{\Sigma(t)}\int_0^\delta|\psi_\eps-\chi_{\Omega}|(y+r\nu(y,t),t)r\dL r\dL \mathcal{H}^{d-1}(y) \\
          &+\int_{\Sigma(t)}\int_0^\delta|\psi_\eps-\chi_{\Omega}|(y-r\nu(y,t),t)r\dL r\dL \mathcal{H}^{d-1}(y)
        \bigg)^2 \\
        =\,\,&C\int_{\Sigma(t)}\int_{-\delta}^\delta|\psi_\eps-\chi_{\Omega(t)}|(y+r\nu(y,t),t)\dist(y+r\nu(x,t),\Sigma(t))\dL r\dL \mathcal{H}^{d-1}(y) \\
        \leq\,\,&C\int_{\mathcal{U}_\delta(t)}|\psi_\eps(x,t)-\chi_{\Omega(t)}(x)|\dist(x,\Sigma(t))\dL x
        \leq\,\,C\mathcal{F}_\eps(t).
      \end{align*}
      In view of~\eqref{eq:pf_gronwall} and the well-preparedness condition~\eqref{item:optimalrate} we obtain Theorem~\ref{thm:ACtoMCF}.
		\end{proof}

    \section{Construction of calibration: Proof of Proposition~\ref{prop:strong_is_calibrated}}

    \begin{proof}[Proof of Proposition~\ref{prop:strong_is_calibrated}]
      Let $(\Sigma(t))_{t\in[0,T]}$ be a smooth solution to volume-preserving mean curvature flow.
      and let $\delta>0$ be sufficiently small such that $\pi_\Sigma$, with the notation of the proof of Theorem~\ref{thm:ACtoMCF},
      is well defined, injective and of class $C^2$ on $\mathcal{U}_{2\delta}$.
      Define a smooth cutoff function $\zeta:\mathbb{R}\rightarrow[0,\infty)$ such that $\zeta(r)=1-r^2$
      for $r<\delta/2$, $\zeta=0$ for $r>\delta$ and define
      \[\xi(\cdot,t)\coloneqq\zeta(\sdist(\cdot,\Sigma(t)))\nabla\sdist(\cdot,\Sigma(t)).\]
      Next, let $\theta$ be a smooth truncation of the identity, i.e.,
      $\theta(r)=-\theta(-r),\,\theta(r)=r$ for $|r|<\delta/2$ and $\theta(r)=\delta$ for $r\geq\delta$.
      Now we define $\vartheta(x,t)\coloneqq\theta(\sdist(x,\Sigma(t)))$.

      Finally we construct the vector field $B$.
      Let $V(\cdot,t)$ denote the normal velocity of the interface $\Sigma(t)=\{\vartheta(\cdot,t)=0\}$ and
      let $\eta$ be a cutoff function such that $\eta(r)=1$ for $|r|<\delta$ and $\eta(r)=0$ for $r>2\delta$.
      Now consider the ansatz
      \[B(x,t)\coloneqq\eta(\sdist(x,\Sigma(t)))((V\nu+X)\circ\pi_{\Sigma(t)}(x))\]
      for some tangent vector field $X(\cdot,t):\Sigma(t)\rightarrow T\Sigma(t)$.
      Then $\nu\otimes\nu:\nabla B=0$ and hence
      \begin{align*}
        \nabla\cdot B
        &=(\mathrm{Id}-\nu\otimes\nu):\nabla B+\nu\otimes\nu:\nabla B \\
        &=\div_{\Sigma(t)}B \\
        &=\div_{\Sigma(t)}(V\nu+X) \\
        &=V\div_{\Sigma(t)}\nu+\div_{\Sigma(t)}X.
      \end{align*}
      We can construct such an $X(\cdot,t)$ by solving the PDE
      \[-\Delta_{\Sigma(t)}\varphi=VH-c\quad\text{on $\Sigma(t)$,}\]
      where $c(t)=\fint_{\Sigma(t)} VH\dL\mathcal{H}^{d-1}$.
      Then the right-hand side satisfies the compatibility condition
      \begin{align*}
        \int_{\Sigma(t)}V H-c\dL\mathcal{H}^{d-1}
        &=0,
      \end{align*}
      and existence and uniqueness of weak solutions in $H^1_{(0)}(\Sigma(t))$
      can easily be shown with the Lax--Milgram lemma, cf.~\cite[Lemma 4]{Kroemer2022}.
      Since $\Sigma(t)$ is $C^{3,\alpha}$ and the normal velocity $V$ is of class $C^{1,\alpha}$, the regularity of $\varphi$ can be improved to $C^{3,\alpha}$
      using Schauder estimates, cf.~\cite[Proof of Thm.\ 1]{Laux2022}.
      Now set $X\coloneqq\nabla_{\Sigma(t)}\varphi$.
      Then $B$ is of class $C^{1,\alpha}$, in particular $C^{0,1}$, and satisfies the required properties:
      \begin{align}
        \label{eq:B1}\nu\otimes\nu:\nabla B&=0\qquad\text{on $\Sigma(t)$,}
        \\ \label{eq:B2}\div B&=c\qquad\text{on $\Sigma(t)$,}
      \end{align}
      and hence by Lipschitz continuity the divergence constraints~\eqref{eq:divB} and~\eqref{eq:xixi_nablaB}.

      Now we compute, on $\Sigma$,
        \begin{align}
          \partial_t\sdist+B\cdot\nabla\sdist
          &=-V+B\cdot\nu
          =-V+V=0.
        \end{align}
        Since both $|\xi|^2=(\zeta\circ\sdist)^2$ and $\vartheta=\theta\circ\sdist$ are functions of the signed distance and Lipschitz,
        we immediately obtain~\eqref{eq:transp_weight} and~\eqref{eq:transp_absxi}.

        It remains to show~\eqref{eq:transp_xi} and~\eqref{eq:extGEE}.
        Since $\zeta'(0)=0$, we have, on $\Sigma$,
        \begin{align*}
          B\cdot\xi+\nabla\cdot\xi-\lambda
          &=B\cdot\nu+|\nabla\sdist|^2\zeta'(0)+\zeta(0)\nabla\cdot\nu-\lambda
          =V+H-\lambda
          =0.
        \end{align*}
        By Lipschitz continuity of $B$ and $\xi$ we get~\eqref{eq:extGEE}.
        Finally we compute
        \begin{align*}
          &\quad(\partial_t\xi+(B\cdot\nabla)\xi+(\nabla B)^\transpose\xi)(\cdot,t) \\
          &=\zeta'(\sdist)(\partial_t\sdist+B\cdot\nabla\sdist)\nabla\sdist
          +\zeta(\sdist)(\partial_t\nabla\sdist+(B\cdot\nabla)\nabla\sdist+(\nabla B)^\transpose\nabla\sdist)
        \end{align*}
        As before, the first term is $O(\dist(\cdot,\Sigma(t)))$. 
        Thus it remains to compute the second term.
        We have, on $\Sigma$,
        \begin{align*}
          0&=\nabla(\partial_t\sdist+(B\cdot\nabla)\sdist) \\
          &=\partial_t\nabla\sdist+(B\cdot\nabla)\nabla\sdist+(\nabla B)^T\nabla\sdist \\
          &=\partial_t\xi+(B\cdot\nabla)\xi+(\nabla B)^T\xi.
        \end{align*}
        This concludes the proof of Proposition~\ref{prop:strong_is_calibrated}.
      \end{proof}

   \section{Relative energy estimate: Proof of Proposition~\ref{prop:AC}}
	
This section is devoted to the proof of the relative energy estimate in Proposition~\ref{prop:AC}.
We will need an appropriate weak formulation of the nonlocal Allen--Cahn equation, which we will later test with the extended velocity field $B$. 
It is easy to check that testing~\eqref{eq:AC} with $B\cdot \eps \nabla u_\eps$
we have for any solution $u_\eps$ of the nonlocal Allen--Cahn equation~\eqref{eq:AC} 
\begin{align}
	&\int (\nabla \cdot B) \Big( \frac\eps2 |\nabla u_\eps|^2 + \frac1\eps W(u_\eps)\Big) \dL x
	-	\int \nu_\eps \cdot \nabla B  \nu_\eps |\nabla \psi_\eps| \dL x 
	\\& =-\int \left(V_\eps-\lambda_\eps\sqrt{2W(u_\eps)}\right) \nu_\eps \cdot B |\nabla u_\eps| \dL x 
	+\int \nu_\eps \cdot \nabla B  \nu_\eps \big(\eps |\nabla u_\eps|^2 - |\nabla\psi_\eps|\big) \dL x,
	\label{eq:weakAC}
\end{align}
where we omitted the domain of integration $\mathbb{R}^d\times\{t\}$, for $t\in(0,T)$, cf.~\cite[Section 3.2]{LauxSimon}.

The following simple lemma, cf.~\cite[Lemma 4]{FischerLauxSimon}, states the basic coercivity properties of the relative energy $\E_\eps$.
		
	\begin{lemma}\label{lem:Ereleps}
	  There exist constants $0<c,C<\infty$ such that
    \begin{align}
      \label{eq:Ereleps1} \int \Big(\sqrt{\eps} |\nabla u_\eps| -\frac1{\sqrt{\eps}} \sqrt{2W(u_\eps)}\Big)^2\dL x 
      &\leq 2 \E_\eps[u_\eps,\Sigma],
      \\ \label{eq:Ereleps2}\int |\nu_\eps -\xi|^2 |\nabla \psi_\eps|\dL x
      &\leq 2 \E_\eps[u_\eps,\Sigma],
      \\ \label{eq:Ereleps3}\int |\nu_\eps -\xi|^2\eps |\nabla u_\eps|^2 \dL x 
      &\leq 12 \E_\eps[u_\eps,\Sigma],
      \\\label{eq:Ereleps4} \int\min\{\dist^2(\cdot,\Sigma),c\}\Big( \frac\eps2 |\nabla u_\eps |^2 +\frac1\eps W(u_\eps) \Big) \dL x
      &\leq C(\Sigma) \E_\eps[u_\eps,\Sigma].
    \end{align}
	\end{lemma}

  Now we are in the position to prove the proposition.

		\begin{proof}[Proof of Proposition~\ref{prop:AC}]
			We compute using Gauss' theorem
			\begin{align}
				\ddt \E_\eps(t) 
				=& \ddt E_\eps[u_\eps(\cdot,t)] + \ddt \int_{\R^d\times\{t\}}  \xi  \cdot \nabla \psi_\eps \dL x
				\\=& -\int_{\R^d\times\{t\}} \frac1{\eps} (\eps \p_tu_\eps)^2 \dL x
				- \int_{\R^d\times\{t\}} (\nabla \cdot \xi) \sqrt{2W(u_\eps)} \p_t u_\eps \dL x
				+ \int_{\R^d\times\{t\}} \p_t \xi \cdot \nabla \psi_\eps \dL x .
			\end{align}
			In the following, we again omit the domain of integration $\R^d\times\{t\}$. 
			We set $V_\eps \coloneqq  \eps \p_t u_\eps$.
			Then we see, using that $\int V_\eps\sqrt{2W(u_\eps)} \dL x =  \ddt \int \psi_\eps \dL x =0$,
			\begin{align*}
					\ddt \E_\eps(t) = \int \left( -\frac1\eps V_\eps^2 -\frac1\eps  V_\eps \big( \nabla\cdot \xi -  \lambda  \big)\sqrt{2W(u_\eps)}  - \p_t\xi \cdot \nu_\eps |\nabla \psi_\eps|  \right)\dL x.
			\end{align*}
			We add the weak formulation~\eqref{eq:weakAC}, tested with the velocity field $B$, to obtain
			\begin{align*}
				\ddt \E_\eps(t) 
				=& \int \left(- \frac1\eps V_\eps^2  -\frac1\eps V_\eps \big(\nabla \cdot\xi   - \lambda\big) \sqrt{2W(u_\eps)}
				+ \left(V_\eps-\lambda_\eps\sqrt{2W(u_\eps)}\right) \nu_\eps \cdot B |\nabla u_\eps|\right) \dL x
				\\&+\int (\nabla \cdot B) \Big( \frac\eps2 |\nabla u_\eps|^2 + \frac1\eps W(u_\eps)\Big) \dL x
				-	\int \nu_\eps \cdot  \nabla B  \nu_\eps |\nabla \psi_\eps| \dL x 
				\\&- \int \nu_\eps \cdot \p_t \xi |\nabla \psi_\eps| \dL x
				\\&-\int \nu_\eps \cdot \nabla B  \nu_\eps \big(\eps |\nabla u_\eps|^2- |\nabla\psi_\eps|\big) \dL x.
			\end{align*}
			Decomposing the vector field $B= (B\cdot \xi)\xi + (\mathrm{Id} -\xi \otimes \xi)B$, completing squares,
      and adding zero to make the transport term for $\xi$ appear, we get				
			\begin{equation}\label{eq:dissipation}
			  \begin{split}	
				\ddt \E_\eps(t)
				&+\frac12 \int \frac1\eps \Big(V_\eps + (\nabla \cdot \xi-  \lambda)\sqrt{2W(u_\eps)} \Big)^2\dL x
				+ \frac12 \int \frac1\eps \Big| V_\eps \nu_\eps - \eps |\nabla u_\eps|(B\cdot \xi) \xi\Big|^2 \dL x
				\\=&\frac12 \int  \Big(( \nabla \cdot \xi-  \lambda )^2 \frac1\eps2W(u_\eps)
				+ (B\cdot \xi )^2 |\xi |^2  \eps |\nabla u_\eps|^2 \Big)\dL x
				\\&+ \int \big(V_\eps \nu_\eps (\mathrm{Id}-\xi\otimes \xi) B -\lambda_\eps \sqrt{2W(u_\eps)} \nu_\eps\cdot B \big) |\nabla u_\eps| \dL x
				\\& +\int \big(\nabla \cdot B -\nu_\eps \cdot \nabla B \nu_\eps + \nu_\eps \cdot (B\cdot \nabla) \xi + \xi (\nu_\eps \cdot \nabla) B\big)  |\nabla \psi_\eps| \dL x 
				\\&-\int \nu_\eps \cdot(\p_t \xi+ (B\cdot \nabla ) \xi +(\nabla B)^\transpose \xi) |\nabla \psi_\eps|\dL x 
				\\&+\int (\nabla \cdot B) \Big( \frac\eps2 |\nabla u_\eps|^2 + \frac1\eps W(u_\eps) - |\nabla \psi_\eps| \Big) \dL x
					\\&-\int \nu_\eps \cdot \nabla B  \nu_\eps \big(\eps |\nabla u_\eps|^2 - |\nabla\psi_\eps|\big) \dL x.
			  \end{split}
			\end{equation}
      Completing another square and 
			using $\xi \otimes \nu_\eps + \nu_\eps \otimes \xi = -(\nu_\eps -\xi)\otimes (\nu_\eps-\xi) + \nu_\eps \otimes \nu_\eps + \xi \otimes \xi$,
			we may write the right-hand side of~\eqref{eq:pfACbeforesymmetry} as
			\begin{align}
				&\notag\frac12 \int  \frac1\eps \Big( ( \nabla \cdot \xi-  \lambda ) \sqrt{2W(u_\eps)}
				+ \eps |\nabla u_\eps| B\cdot \xi \Big)^2 \dL x
        \\&\notag+\frac12 \int \big(|\xi|^2-1\big) (B\cdot \xi)^2 \eps |\nabla u_\eps|^2 \dL x
				\\&\notag-\int (\nabla \cdot \xi -\lambda) B\cdot \xi |\nabla \psi_\eps|\dL x
				\\&\notag+ \int \big(V_\eps \nu_\eps \cdot (\mathrm{Id}-\xi\otimes \xi) B -\lambda_\eps \sqrt{2W(u_\eps)} \nu_\eps\cdot B\big)|\nabla u_\eps|  \dL x
				\\&\notag+\int (\nabla \cdot B) (1-\xi \cdot \nu_\eps) |\nabla \psi_\eps| \dL x + \int (\nabla \cdot B) \xi \cdot \nu_\eps |\nabla \psi_\eps| \dL x
				\\&\notag- \int \nabla B \colon (\nu_\eps- \xi) \otimes( \nu_\eps-\xi) |\nabla \psi_\eps|\dL x
				\\&\notag-\int \nu_\eps \cdot (\xi \cdot \nabla ) B  |\nabla \psi_\eps|\dL x+ \int \nu_\eps \cdot (B\cdot \nabla) \xi |\nabla \psi_\eps|\dL x
				\\&\notag-\int (\nu_\eps-\xi) \cdot \left( \p_t \xi +(B\cdot \nabla) \xi + (\nabla B)^{\transpose}  \xi \right) |\nabla \psi_\eps| \dL x
				\\& \notag-\int \xi \cdot \left( \p_t \xi +(B\cdot \nabla) \xi  \right) |\nabla \psi_\eps| \dL x
				\\&\notag+\int (\nabla \cdot B) \Big( \frac\eps2 |\nabla u_\eps|^2 + \frac1\eps W(u_\eps) - |\nabla \psi_\eps| \Big) \dL x
				\\&-\int \nu_\eps \cdot \nabla B  \nu_\eps \big(\eps |\nabla u_\eps|^2 - |\nabla\psi_\eps|\big) \dL x.
				\label{eq:pfACbeforesymmetry}
			\end{align}
			Two integrations by parts and the symmetry of the Hessian $\nabla^2\psi_\eps$ imply
			\begin{align*}
			\int   (B\cdot \nabla) \xi  \cdot \nabla \psi_\eps  \dL x 
			=\int (\xi \cdot \nabla )B \cdot \nabla \psi_\eps \dL x 
			+ \int \big( (\nabla \cdot \xi) B - (\nabla \cdot B) \xi \big) \cdot \nabla \psi_\eps \dL x.
			\end{align*}
			Combining this with $\nabla \psi_\eps = -\nu_\eps |\nabla \psi_\eps| $, we may again replace three terms in~\eqref{eq:pfACbeforesymmetry}
			by the term $(\nabla \cdot \xi) B\cdot \nu_\eps |\nabla \psi_\eps|$ so that we get
			\begin{align}
				\ddt &\notag\E_\eps(t)
			  +\frac12 \int \frac1\eps \Big(V_\eps + (\nabla \cdot \xi-  \lambda)\sqrt{2W(u_\eps)} \Big)^2\dL x
			  +\frac12 \int \frac1\eps \Big| V_\eps \nu_\eps - \eps |\nabla u_\eps|(B\cdot \xi) \xi\Big|^2 \dL x
				\\\notag=&\frac12 \int  \frac1\eps \Big( ( \nabla \cdot \xi-  \lambda ) \sqrt{2W(u_\eps)}
				+ \eps |\nabla u_\eps| B\cdot \xi \Big)^2 \dL x
				+\frac12 \int \big(|\xi|^2-1\big) (B\cdot \xi)^2 \eps |\nabla u_\eps|^2 \dL x
				\\\notag&-\int (\nabla \cdot \xi -\lambda) (1-\xi \cdot \nu_\eps) B\cdot \xi |\nabla \psi_\eps|\dL x
				+\int (\lambda-\lambda_\eps) \nu_\eps \cdot B |\nabla \psi_\eps| \dL x
				\\\notag&+ \int \Big(V_\eps +(\nabla \cdot \xi -\lambda)\sqrt{2W(u_\eps)} \Big) \nu_\eps \cdot (\mathrm{Id}-\xi\otimes \xi) B |\nabla u_\eps|\dL x
				\\\notag&+\int (\nabla \cdot B) (1-\xi \cdot \nu_\eps) |\nabla \psi_\eps| \dL x 
				- \int(\nu_\eps- \xi) \cdot \nabla B ( \nu_\eps-\xi) |\nabla \psi_\eps|\dL x
				\\\notag&-\int (\nu_\eps-\xi) \cdot \left( \p_t \xi +(B\cdot \nabla) \xi + (\nabla B)^{\transpose}  \xi \right) |\nabla \psi_\eps| \dL x
				\\\notag& -\frac12 \int \left( \p_t |\xi|^2 +(B\cdot \nabla) |\xi|^2  \right) |\nabla \psi_\eps| \dL x
				\\\notag&+\int (\nabla \cdot B) \Big( \frac\eps2 |\nabla u_\eps|^2 + \frac1\eps W(u_\eps) - |\nabla \psi_\eps| \Big) \dL x
				\\ &-\int \nu_\eps \cdot \nabla B  \nu_\eps \big(\eps |\nabla u_\eps|^2 - |\nabla\psi_\eps|\big) \dL x.
				\label{eq:relEneps}
			\end{align}
			We argue term-by-term that the right-hand side can be  controlled suitably.
			By and large, the argument is similar to the one in the sharp-interface case in~\cite{Laux2022},
      here based on the coercivity properties of $\E_\eps$ collected in Lemma~\ref{lem:Ereleps}.
      Let us first estimate the terms that are analogous to~\cite{Laux2022}.
			For the first term, by Young's inequality we have
			\begin{align*}
				&\frac1{2\eps}  \Big( ( \nabla \cdot \xi-  \lambda ) \sqrt{2W(u_\eps)}
				+ \eps |\nabla u_\eps| B\cdot \xi \Big)^2 
				\\&\leq    ( \nabla \cdot \xi-  \lambda +B\cdot \xi )^2 \eps |\nabla u_\eps|^2
				+  (\nabla \cdot \xi -\lambda)^2 \Big(\sqrt{\eps} |\nabla u_\eps| -\frac1{\sqrt{\eps}} \sqrt{2W(u_\eps)}\Big)^2.
			\end{align*}
			The contribution of these terms are controlled by $\E_\eps(t)$ using~\eqref{eq:extGEE} and~\eqref{eq:Ereleps4}, respectively,~\eqref{eq:Ereleps1}. 
			The second term in~\eqref{eq:relEneps} is controlled by~\eqref{eq:xishort} in conjunction with~\eqref{eq:Ereleps4}. 
			The third term is directly controlled by $\|\nabla \cdot \xi -\lambda\|_\infty \|B\cdot \xi\|_\infty \E_\eps(t)$.
      The analogous argument holds for the sixth term.
      For the fifth term we use Young's inequality:
      \begin{align*}
        &\int\left(V_\eps+(\nabla\cdot\xi-\lambda)\sqrt{2W(u_\eps)}\right)\nu_\eps\cdot(\mathrm{Id}-\xi\otimes\xi)B|\nabla u_\eps|\dL x \\
        &\leq\frac{1}{4}\int\frac{1}{\eps}\left(V_\eps+(\nabla\cdot\xi-\lambda)\sqrt{2W(u_\eps)}\right)^2\dL x
        +\int(\nu_\eps\cdot(\mathrm{Id}-\xi\otimes\xi)B)^2\eps|\nabla u_\eps|^2\dL x \\
        &\leq\frac{1}{4}\int\frac{1}{\eps}\left(V_\eps+(\nabla\cdot\xi-\lambda)\sqrt{2W(u_\eps)}\right)^2\dL x
        +\|B\|_\infty^2\int|\nu_\eps-(\nu_\eps\cdot\xi)\xi|^2\eps\nabla u_\eps|^2\dL x.
      \end{align*}
      The first term is absorbed in the first term on the left-hand side of~\eqref{eq:relEneps}.
      The second term is estimated by~\eqref{eq:Ereleps3}.
      The seventh term is controlled by~\eqref{eq:Ereleps2}, since $(\nu_\eps-\xi)\cdot\nabla B(\nu_\eps-\xi)\leq\|\nabla B\|_\infty|\nu_\eps-\xi|^2$.
      For the eighth term we have, using~\eqref{eq:extGEE} and Young's inequality,
      \begin{align*}
        &(\nu_\eps-\xi)\cdot\left( \p_t \xi +(B\cdot \nabla) \xi + (\nabla B)^{\transpose}  \xi \right)|\nabla \psi_\eps| \\
        \leq\,\,&\frac{1}{2}|\nu_\eps-\xi|^2|\nabla\psi_\eps|+\frac{1}{2}C\min\{\dist^2(\cdot,\Sigma(t)),c\}|\nabla\psi_\eps|.
      \end{align*}
      Since $|\nabla\psi_\eps|\leq\frac{1}{2}\eps|\nabla u_\eps|^2+\frac{1}{\eps}W(u_\eps)$,
      the eighth term is controlled by~\eqref{eq:Ereleps2} and~\eqref{eq:Ereleps4}.
      The ninth term is controlled by~\eqref{eq:transp_absxi}.
      The second to last term is controlled by $\|\nabla\cdot B\|_\infty\mathcal{E}_\eps(t)$.
      Thus is remains to estimate the fourth term and the last term.
      
      For the last term in~\eqref{eq:relEneps} we observe that, using $|\nu_\eps\cdot\nabla B\nu_\eps-\xi\cdot\nabla B\xi|\leq\|\nabla B\|_\infty|\nu_\eps-\xi|$
      and Young's inequality,
      \begin{align*}
        &\quad\int \nu_\eps \cdot \nabla B  \nu_\eps \big(\eps |\nabla u_\eps|^2 - |\nabla\psi_\eps|\big) \dL x \\
        &\leq\int\xi\cdot\nabla B\xi\big(\eps |\nabla u_\eps|^2 - |\nabla\psi_\eps|\big) \dL x \\
        &\quad+\|\nabla B\|_\infty\int|\nu_\eps-\xi|\sqrt{\eps}|\nabla u_\eps|\left(\sqrt{\eps}|\nabla u_\eps|-\frac{1}{\sqrt{\eps}}\sqrt{2W(u_\eps)}\right)\dL x \\
        &\leq\int\xi\cdot\nabla B\xi\big(\eps |\nabla u_\eps|^2 - |\nabla\psi_\eps|\big) \dL x
        +\|\nabla B\|_\infty\int|\nu_\eps-\xi|^2\eps|\nabla u_\eps|^2\dL x \\
        &\quad+\|\nabla B\|_\infty\int\left(\sqrt{\eps}|\nabla u|-\frac{1}{\sqrt{\eps}}\sqrt{2W(u_\eps)}\right)^2\dL x.
      \end{align*}
      Here, the last two terms are bounded by~\eqref{eq:Ereleps1} and ~\eqref{eq:Ereleps3}, respectively. 
      We compute, using Young's inequality,
      \begin{align*}
        \int\xi\otimes\xi:\nabla B(\eps|\nabla u_\eps|^2-|\nabla\psi_\eps|)\dL x
        \leq\,\,&\frac{1}{2}\int\left(\xi\otimes\xi:\nabla B\right)^2\eps|\nabla u|^2\dL x \\
                &+\frac{1}{2}\int\left(\sqrt{\eps}|\nabla u_\eps|-\frac{1}{\sqrt{\eps}}\sqrt{2W(u_\eps)}\right)^2\dL x.
      \end{align*}
      Using the coercivity estimate~\eqref{eq:weight-bilip-estimate} and~\eqref{eq:Ereleps1},
      the second summand is bounded by $\mathcal{E}_\eps$.
      For the first term we have, by~\eqref{eq:xixi_nablaB},
      \begin{align*}
        \frac{1}{2}\int\left(\xi\otimes\xi:\nabla B\right)^2\eps|\nabla u_\eps|^2\dL x
        &\leq\frac{1}{2}\|\nabla B\|_\infty^2\int_{\supp\xi}|\sdist|^2\eps|\nabla u|^2\dL x
      \end{align*}
      which is bounded by $\mathcal{E}_\eps$ by~\eqref{eq:Ereleps4}.

      Next we estimate the fourth term in~\eqref{eq:relEneps}.
      Since by Gauss' theorem
      \[\int_{\Omega(t)}\nabla\cdot B\dL x
        =\int_{\Sigma(t)} B\cdot\nu\dL\mathcal{H}^{d-1}=\int_{\Sigma(t)}V\dL\mathcal{H}^{d-1}=\ddt|\Omega(t)|=0,
      \]
      we have
      \begin{align*}
        \int(\lambda-\lambda_\eps)\nu_\eps\cdot B|\nabla\psi_\eps|\dL x
        &=-(\lambda-\lambda_\eps)\int(\nabla\cdot B)\psi_\eps\dL x \\
        &=-(\lambda-\lambda_\eps)\int(\nabla\cdot B)(\psi_\eps-\chi_{\Omega(t)})\dL x.
      \end{align*}
      Furthermore, since $\ddt\int\psi_\eps\dL x=0=\ddt|\Omega(t)|$,
      we also have
      \begin{align*}
        &\quad\int(\nabla\cdot B)(\psi_\eps-\chi_{\Omega})\dL x \\
        &=\int(\nabla\cdot B-c)(\psi_\eps-\chi_{\Omega})\dL x
        +\int c(\psi_\eps-\chi_{\Omega})\dL x \\
        &=\int(\nabla\cdot B-c)(\psi_\eps-\chi_{\Omega})\dL x
        +c\left(\int_{\mathbb{R}^d}\psi_\eps(x,t)\dL x-|\Omega(t)|\right) \\
        &=\int(\nabla\cdot B-c)(\psi_\eps-\chi_{\Omega})\dL x
        +c\left(\int_{\mathbb{R}^d}\psi_\eps(x,0)\dL x-|\Omega(0)|\right).
      \end{align*}
      By the well-preparedness assumption~\eqref{item:mass}, the second summand vanishes.
      By~\eqref{eq:divB} we have
      \begin{align*}
        (|\lambda|+|\lambda_\eps|)\int(\nabla\cdot B-c)(\psi_\eps-\chi_{\Omega})\dL x
        &\leq C(|\lambda|+|\lambda_\eps|)\int|\vartheta||\psi_\eps-\chi_{\Omega}|\dL x \\
        &=C(|\lambda|+|\lambda_\eps|)\mathcal{F}_\eps.
      \end{align*}
      Therefore we have in total
      \begin{align}\label{eq:finalestimateEeps}
        \notag\ddt\E_\eps(t)
        +\frac14 \int \frac1\eps \Big(V_\eps + (\nabla \cdot \xi-  \lambda)\sqrt{2W(u_\eps)} \Big)^2\dL x
        &+ \frac12 \int \frac1\eps \Big| V_\eps \nu_\eps - \eps |\nabla u_\eps|(B\cdot \xi) \xi\Big|^2 \dL x \\ 
        &\leq C(\E_\eps(t)+\F_\eps(t)).
      \end{align}

      Finally we estimate $\ddt\F_\eps$ by decomposing it into a term which is bounded by $\mathcal{E}_\eps+\mathcal{F}_\eps$
      and a small dissipation term which can be absorbed on the left-hand side of~\eqref{eq:finalestimateEeps}.

      We smuggle in
      $\int(B\cdot\nabla\vartheta)(\psi_\eps-\chi_{\Omega})\dL x
      =-\int\vartheta B\cdot\nabla\psi_\eps\dL x-\int(\nabla\cdot B)\vartheta(\psi_\eps-\chi_{\Omega})\dL x$
      and obtain
      \begin{align*}
        \ddt\F_\eps(t)
        &=\int\partial_t\vartheta(\psi_\eps-\chi_{\Omega})\dL x
        +\int\vartheta\partial_t\psi_\eps\dL x-\int_{\Sigma}\vartheta V\dL\mathcal{H}^{d-1} \\
        &=\int(\partial_t\vartheta+B\cdot\nabla\vartheta)(\psi_\eps-\chi_{\Omega})\dL x
        +\int(\nabla\cdot B)\vartheta(\psi_\eps-\chi_{\Omega})\dL x \\
        &\quad+\int\vartheta(\partial_t\psi_\eps+B\cdot\nabla\psi_\eps)\dL x.
      \end{align*}
      Since $(\partial_t\vartheta+B\cdot\nabla\vartheta)=O(\dist(\cdot,\Sigma))$ and $B$ is Lipschitz,
      the first two summands are bounded by $\F_\eps$.
      It only remains to estimate the last integral, which amounts to estimating the error in the transport equation for $\psi_\eps$.

      Indeed, decomposing the vector field $B=(B\cdot \xi)\xi + (\mathrm{Id}-\xi \otimes \xi)B$ once more and applying Young's inequality, we compute
      \begin{align*}
        \int\vartheta(\partial_t\psi_\eps+B\cdot\nabla\psi_\eps)\dL x
        &\leq\int\vartheta\left(
          \frac{1}{\eps}\sqrt{2W(u_\eps)}V_\eps-\frac{1}{\eps}\sqrt{2W(u_\eps)}\eps|\nabla u_\eps|\nu_\eps\cdot(B\cdot\xi)\xi
        \right)\dL x \\
		    &\quad+\int\vartheta|\nabla\psi_\eps|B\cdot(\nu_\eps-(\xi\cdot\nu_\eps)\xi)\dL x \\
        &\leq\int\vartheta\frac{1}{\eps}\sqrt{2W(u_\eps)}\left(
          V_\eps-\eps|\nabla u_{\eps}|\nu_\eps\cdot(B\cdot\xi)\xi
        \right)\dL x \\
		    &\quad+\|B\|_\infty\int|\vartheta||\nu_\eps-(\nu_\eps\cdot\xi)\xi||\nabla\psi_\eps|\dL x \\
        &\leq 2\int\vartheta^2\frac{1}{\eps}W(u_\eps)\dL x
        +\frac{1}{4}\int\frac{1}{\eps}\left(
          V_\eps-\eps|\nabla u_{\eps}|\nu\cdot(B\cdot\xi)\xi
        \right)^2\dL x \\
		    &\quad+\|B\|_\infty\int|\vartheta||\nu_\eps-(\nu_\eps\cdot\xi)\xi||\nabla\psi_\eps|\dL x.
      \end{align*}
      The first term is estimated by~\eqref{eq:Ereleps4}.
      The second term is absorbed in the dissipation~\eqref{eq:finalestimateEeps}
      after adding $\ddt\mathcal{E}_\eps$ and $\ddt\mathcal{F}_\eps$ together.
      For the last term we apply Young's inequality once more to obtain
      \begin{align*}
        \int|\vartheta||\nu_\eps-(\nu_\eps\cdot\xi)\xi||\nabla\psi_\eps|\dL x
        &\leq\frac{1}{2}\int\vartheta^2|\nabla\psi_\eps|\dL x
        +\int(1-\nu_\eps\cdot\xi)|\nabla\psi_\eps|\dL x \\
        &\leq\frac{1}{2}\int\vartheta^2\left(\frac{\eps}{2}|\nabla u_\eps|^2+\frac{1}{\eps}W(u_\eps)\right)\dL x
        +\mathcal{E}_\eps.
      \end{align*}
      This is again estimated by $\mathcal{E}_\eps(t)$.
      Therefore
      \[\ddt\left(\mathcal{E}_\eps(t)+\mathcal{F}_\eps(t)\right)
        \leq C\left(\mathcal{E}_\eps(t)+\mathcal{F}_\eps(t)\right).\qedhere
      \]
	\end{proof}

  Finally we give a short proof of Lemma~\ref{lem:existence_init_data}.

  \begin{proof}[Proof of Lemma~\ref{lem:existence_init_data}]
    The proof is similar to the unconstrained case in~\cite{FischerLauxSimon}.
    For $a\in\mathbb{R}$ define $u_\eps^{a}(x)\coloneqq U(-\eps^{-1}\sdist(x)-a)$, where $U$ is as in~\eqref{eq:init_data_optimal_profile},
    and let $\psi_{\eps}^a\coloneqq\phi\circ u_\eps^a$.
    Since $a\mapsto\int\psi_\eps^a\dL x$ is continuous and $\int\psi_\eps^a\dL x\rightarrow 0$ as $a\rightarrow+\infty$
    and $\int\psi_\eps^a\dL x\rightarrow\infty$ as $a\rightarrow-\infty$,
    there exists, for each $\eps>0$, $a_\eps\in\mathbb{R}$ such that $\int\psi_\eps^{a_\eps}\dL x=|\Omega(0)|$.
    Furthermore
    \begin{align*}
      \frac{d}{da}\bigg|_{a=0}\int\phi\circ u_\eps^a\dL x
      &=\int\frac{1}{\eps}\sqrt{2W(u_\eps^0)}U'(-\eps^{-1}\sdist)\dL x
      =\int\frac{2}{\eps}W(u_\eps^0)\dL x \\
      &=E_\eps(u_\eps^0)\rightarrow\mathcal{H}^{d-1}(\Sigma(0))\neq 0,
    \end{align*}
    and $\int\phi\circ U(-\eps^{-1}\sdist)\dL x=|\Omega(0)|(1+O(\eps))$.
    Hence $a_\eps=O(\eps)$.
    For simplicity we write $u_\eps=u_\eps^{a_\eps}$.
    Now we compute, using $U'(s)=\sqrt{2W(U(s))}$ and $1-\nabla\sdist\cdot\xi\leq 1-|\xi|^2\leq c\dist^2(\cdot,\Sigma(0))$,
    \begin{align*}
      \mathcal{E}_\eps(0)
      &=\int\left(\frac{\eps}{2}|\nabla u_\eps|^2+\frac{1}{\eps}W(u_\eps)\right)\dL x
      +\int\xi\cdot\nabla\psi_\eps\dL x \\
      &=\int\left(\frac{1}{2\eps}\left|U'\left(-\eps^{-1}\sdist(x)-a_\eps\right)\right|^2+\frac{1}{\eps}W(u_\eps(x))\right)\dL x
      -\int\frac{2}{\eps}W(u_\eps)\xi\cdot\nabla\sdist(x)\dL x \\
      &=\int(1-\nabla\sdist\cdot\xi)\frac{2}{\eps}W(u_\eps)\dL x \\
      &\leq c\eps^2\int\left(\frac{\dist(x,\Sigma(0))}{\eps}\right)^2\frac{2}{\eps}W(u_\eps)\dL x.
    \end{align*}
    Hence $\E_\eps(0)=O(\eps^2)$.
    The bulk error
    \begin{align*}
      \mathcal{F}_\eps(0)
      &=\int\vartheta(x)(\phi(U(-\eps^{-1}\sdist(x)-a_\eps))-\chi_{\Omega(0)}(x))\dL x
    \end{align*}
    is also $O(\eps^2)$, since $c\dist(\cdot,\Sigma(0))\leq|\vartheta|\leq C\dist(\cdot,\Sigma(0))$
    and $U(s)\rightarrow 0$ as $s\rightarrow-\infty$ and $U(s)\rightarrow 1$ as $s\rightarrow+\infty$.
    Hence $u_\eps$ satisfies condition~\eqref{item:optimalrate}.
  \end{proof}
		
%%%%%%%%%%%%%%%%%%%%%%%%%%%%%%%%%%%%%%%%%%%%%%%%%%%%%%%%%%%%%%%%%%%%%%%%%%%%%%%%%%%%%%%%%%%%%%%%%%%%%%%%%%%%%%%%
\section*{Acknowledgments}
	This project has received funding from the Deutsche Forschungsgemeinschaft (DFG, German Research Foundation) under Germany's Excellence Strategy -- EXC-2047/1 -- 390685813.
%%%%%%%%%%%%%%%%%%%%%%%%%%%%%%%%%%%%%%%%%%%%%%%%%%%%%%%%%%%%%%%%%%%%%%%%%%%%%%%%%%%%%%%%%%%%%%%%%%%%%%%%%%%%%%%%
\frenchspacing
\bibliographystyle{abbrv}
\bibliography{lit}
	
  \end{document}